\documentclass[a4paper,twoside,11pt]{article}

\usepackage{a4wide, fancyhdr, amsmath, amssymb, mathtools, yfonts, bbm}
\usepackage{mathrsfs}
\usepackage{graphicx}
\usepackage{tikz}
\usepackage[all]{xy}
\usepackage[utf8]{inputenc}
\usepackage{amsthm}
\usepackage[english]{babel}
\usepackage{chngcntr}
\usepackage{ifthen}
\usepackage{calc}
\usepackage{hyperref}
\usepackage{authblk}
\numberwithin{equation}{section}
\providecommand{\keywords}[1]{{2010}\textit{ Mathematics Subject Classification. }#1}

\newcommand{\Z}{\mathbb{Z}}
\newcommand{\Q}{\mathbb{Q}}

\newcommand\FF{\mathbb{F}}

\newcommand\Norm{\mathrm{N}}

\newtheorem{lemma}{Lemma}[section]
\newtheorem{theorem}[lemma]{Theorem}
\newtheorem{prop}[lemma]{Proposition}

\title{\vspace{-\baselineskip}\sffamily\bfseries The $16$-rank of $\Q(\sqrt{-p})$}
\author[1]{Peter Koymans\thanks{Niels Bohrweg 1, 2333 CA Leiden, Netherlands, p.h.koymans@math.leidenuniv.nl}}
\affil[1]{Mathematisch Instituut, Leiden University}

\date{\today}

\begin{document}
\maketitle

\begin{abstract}
Recently, a density result for the $16$-rank of $\text{Cl}(\Q(\sqrt{-p}))$ was established when $p$ varies among the prime numbers, assuming a short character sum conjecture. In this paper we prove the same density result unconditionally.
\end{abstract}

\section{Introduction}
If $K$ is a quadratic number field with narrow class group $\text{Cl}(K)$, there is an explicit description of $\text{Cl}(K)[2]$ due to Gauss. Since then the class group of quadratic number fields has been extensively studied. If one is interested in the $2$-part of the class group, i.e. $\text{Cl}(K)[2^\infty]$, the explicit description of $\text{Cl}(K)[2]$ is often very useful. It is for this reason that our current understanding of the $2$-part of the class group is much better than the $p$-part for odd $p$.

In 1984, Cohen and Lenstra put forward conjectures regarding the average behavior of the class group $\text{Cl}(K)$ of imaginary and real quadratic fields $K$. Despite significant effort, there has been relatively little progress in proving these conjectures. Almost all major results are about the $2$-part with the most notable exception being the classical result of Davenport and Heilbronn \cite{DH} regarding the distribution of $\text{Cl}(K)[3]$. Very little is known about $\text{Cl}(K)[p]$ for $p > 3$. The non-abelian version of Cohen-Lenstra has recently also attracted great interest, see \cite{Alberts}, \cite{AK}, \cite{Klys} and \cite{Wood}.

Gerth \cite{Gerth1} studied the distribution of $2\text{Cl}(K)[4]$, when the number of prime factors of the discriminant of $K$ is fixed. Fouvry and Kl\"uners \cite{FK2} computed all the moments of $2\text{Cl}(K)[4]$, when $K$ varies among imaginary or real quadratic fields. In the paper \cite{FK1}, they deduced the probability that the $4$-rank of a quadratic field has a given value. Their work was based on earlier ideas of Heath-Brown \cite{HB}.

The study of $\text{Cl}(K)[2^\infty]$ has often been conducted through the lens of \emph{governing fields}. Let $k \geq 1$ be an integer and let $d$ be an integer with $d \not \equiv 2 \bmod 4$. For a finite abelian group $A$ we define the $2^k$-rank of $A$ to be $\text{rk}_{2^k} \ A := \dim_{\FF_2} 2^{k - 1} A/2^k A$. Then a governing field $M_{d, k}$ is a normal field extension of $\Q$ such that 
\[
\text{rk}_{2^k} \text{Cl}\left(\Q\left(\sqrt{dp}\right)\right)
\]
is determined by the splitting of $p$ in $M_{d, k}$. Cohn and Lagarias \cite{CohnLag} were the first to define the concept of a governing field, and conjectured that they always exist.

If $k \leq 3$, then governing fields are known to exist for all values of $d$. In case $k = 2$ this follows from work of R\'edei \cite{Redei} and Stevenhagen dealt with the case $k = 3$ \cite{Ste1}. The topic was recently revisited by Smith \cite{Smith1}, who found a very explicit description for $M_{d, 3}$ for most values of $d$. He then used this description to prove density results for $4\text{Cl}(K)[8]$ assuming GRH. Not much later Smith \cite{Smith2} introduced \emph{relative governing fields}, which allowed him to prove the most impressive result that $2\text{Cl}(K)[2^\infty]$ has the expected distribution when $K$ varies among all imaginary quadratic fields.

If we let $P(d, k)$ be the statement that a governing field $M_{d, k}$ exists, then there is currently not a single value of $d$ for which the truth or falsehood of $P(d, 4)$ is known. This has been the most significant obstruction in proving density results for the $16$-rank in thin families of the shape $\left\{\Q\left(\sqrt{dp}\right)\right\}_{p \text{ prime}}$.

This barrier was first broken by Milovic \cite{Milovic2}, who dealt with the $16$-rank in the family $\left\{\Q\left(\sqrt{-2p}\right)\right\}_{p \equiv -1 \bmod 4}$. Milovic proves his density result with Vinogradov's method, and does not rely on the existence of a governing field. His use of Vinogradov's method was inspired by work of Friedlander et al. \cite{FIMR, FIMRE}, which is based on earlier work of Friedlander and Iwaniec \cite{FI1}.

Milovic and the author established density results for the families $\left\{\Q\left(\sqrt{-2p}\right)\right\}_{p \equiv 1 \bmod 4}$ and $\left\{\Q\left(\sqrt{-p}\right)\right\}_{p}$, see respectively \cite{KM2} and \cite{KM1} with the latter work being conditional on a short character sum conjecture. Both \cite{KM2} and \cite{KM1} follow the ideas of \cite{FIMR, FIMRE} closely in their treatment of the sums of type I, see Section \ref{sSieve} for a definition. However, if one applies the method of \cite{FIMR, FIMRE} to a number field of degree $n$, one is naturally lead to consider character sums of modulus $q$ and length $q^{\frac{1}{n}}$.

In \cite{KM2} we apply the method from \cite{FIMR, FIMRE} to a number field of degree $4$. This leads to character sums just outside the range of the Burgess bound. Fortunately, the lemmas in Section 3.2 of \cite{KM2} allow us to reduce the size of the modulus from $q$ to $q^{\frac{1}{2}}$, and this enables us to deal with the sums of type I unconditionally. In \cite{KM1} we use a criterion for the $16$-rank of $\Q(\sqrt{-p})$ due to Bruin and Hemenway \cite{BH}, and this criterion is stated most naturally over $\Q\left(\zeta_8, \sqrt{1 + i}\right)$, which has degree $8$. The resulting character sums are far outside the reach of the Burgess bound and we resort to assuming a short character sum conjecture, see \cite[p.\ 8]{KM1}.

In this paper we manage to deal with the $16$-rank of $\Q(\sqrt{-p})$ unconditionally by using a criterion of Leonard and Williams \cite{LW82}, which one can naturally state over $\Q(\zeta_8)$. However, the Leonard and Williams criterion has the significant downside that it is the product of two residue symbols instead of one residue symbol, namely a quadratic and a quartic residue symbol. The resulting sums of type I can still not be treated unconditionally with the method from \cite{FIMR, FIMRE}. Instead, we use a rather ad hoc argument to deal with the resulting character sum.

\begin{theorem}
\label{tMain}
Let $h(-p)$ be the class number of $\Q(\sqrt{-p})$. Then
\[
\lim_{X \rightarrow \infty} \frac{|\{p \textup{ prime} : p \leq X \textup{ and } 16 \mid h(-p)\}|}{|\{p \textup{ prime} : p \leq X\}|} = \frac{1}{16}.
\]
\end{theorem}

Milovic \cite{Milovic1} has previously shown that there are infinitely many primes $p$ with $16$ dividing $h(-p)$. Theorem \ref{tMain} gives an affirmative answer to conjectures in both \cite{CohnLag2} and \cite{Ste2}. For $p$ a prime number, we define $e_p$ by
\begin{align}
\label{edep}
e_p := 
\left\{
	\begin{array}{ll}
		1  & \mbox{if } 16 \mid h(-p) \\
		-1 & \mbox{if } 8 \mid h(-p), 16 \nmid h(-p) \\
		0 & \mbox{otherwise.}
	\end{array}
\right.
\end{align}
Theorem \ref{tMain} is an immediate consequence of the following theorem.

\begin{theorem}
\label{tCancel}
We have
\[
\sum_{p \leq X} e_p \ll \frac{X}{\exp\left(\left(\log X\right)^{0.1}\right)}.
\]
\end{theorem}

It is natural to wonder if the other conditional results in \cite{KM1} can be proven unconditionally using the methods from this paper. This is likely to be the case, but it would require some effort to obtain suitable algebraic results similar to the Leonard and Williams \cite{LW82} criterion used in this paper.

We believe that the ideas introduced by Smith \cite{Smith2} do not apply to the thin families that we deal with here. Indeed, in Smith's paper a crucial ingredient for both the algebraic and analytic part is the fact that a typical integer $N$ has roughly $\log \log N$ prime divisors and that $\log \log N$ goes to infinity as $N$ goes to infinity.

\subsection*{Acknowledgements}
I am very grateful to Djordjo Milovic for his support during this project. I would also like to thank Jan-Hendrik Evertse for proofreading.

\section{Preliminaries}
\subsection{Quadratic and quartic reciprocity}
Let $K$ be a number field with ring of integers $O_K$. We say that an ideal $\mathfrak{n}$ of $O_K$ is odd if $(\mathfrak{n}, 2) = (1)$. Similarly, we say that an element $w$ of $O_K$ is odd if the ideal generated by $w$ is odd. If $\mathfrak{p}$ is an odd prime ideal of $O_K$ and $\alpha \in O_K$, we define the quadratic residue symbol
\[
\left(\frac{\alpha}{\mathfrak{p}}\right)_{2, K} := 
\left\{
	\begin{array}{ll}
		1  & \mbox{if } \alpha \not \in \mathfrak{p} \text{ and } \alpha \equiv \beta^2 \bmod \mathfrak{p} \text{ for some } \beta \in O_K \\
		-1 & \mbox{if } \alpha \not \in \mathfrak{p} \text{ and } \alpha \not \equiv \beta^2 \bmod \mathfrak{p} \text{ for all } \beta \in O_K \\
		0 & \mbox{if } \alpha \in \mathfrak{p}.
	\end{array}
\right.
\]
Then Euler's criterion states
\[
\left(\frac{\alpha}{\mathfrak{p}}\right)_{2, K} \equiv \alpha^{\frac{\Norm(\mathfrak{p}) - 1}{2}} \bmod \mathfrak{p}.
\]
For a general odd ideal $\mathfrak{n}$ of $O_K$, we define
\[
\left(\frac{\alpha}{\mathfrak{n}}\right)_{2, K} := \prod_{\mathfrak{p}^e \parallel \mathfrak{n}} \left(\left(\frac{\alpha}{\mathfrak{p}}\right)_{2, K}\right)^e.
\]
Furthermore, for odd $\beta \in O_K$ we set
\[
\left(\frac{\alpha}{\beta}\right)_{2, K} := \left(\frac{\alpha}{(\beta)}\right)_{2, K}.
\]
We say that an element $\alpha \in K$ is totally positive if for all embeddings $\sigma$ of $K$ into $\mathbb{R}$ we have $\sigma(\alpha) > 0$. In particular, all elements of a totally complex number field are totally positive. We will make extensive use of the law of quadratic reciprocity.

\begin{theorem}
\label{t2R}
Let $\alpha, \beta \in O_K$ be odd. If $\alpha$ or $\beta$ is totally positive, we have
\[
\left(\frac{\alpha}{\beta}\right)_{2, K} = \mu(\alpha, \beta) \left(\frac{\beta}{\alpha}\right)_{2, K},
\]
where $\mu(\alpha, \beta) \in \{\pm 1\}$ depends only on the congruence classes of $\alpha$ and $\beta$ modulo $8$.
\end{theorem}

\begin{proof}
This follows from Lemma 2.1 of \cite{FIMR, FIMRE}.
\end{proof}

If $K = \Q$, we shall drop the subscript. In this case the symbol $\left(\frac{\cdot}{\cdot}\right)$ is to be interpreted as the Kronecker symbol, which is an extension of the quadratic residue symbol to allow for even arguments in the bottom. We presume that the reader is familiar with the quadratic reciprocity law for the Kronecker symbol. Now let $K$ be a number field containing $\Q(i)$ still with ring of integers $O_K$. For $\alpha \in O_K$ and $\mathfrak{p}$ an odd prime ideal of $O_K$, we define the quartic residue symbol $(\alpha/\mathfrak{p})_{4, K}$ to be the unique element in $\{\pm 1, \pm i, 0\}$ such that
\[
\left(\frac{\alpha}{\mathfrak{p}}\right)_{4, K} \equiv \alpha^{\frac{\Norm(\mathfrak{p}) - 1}{4}} \bmod \mathfrak{p}.
\]
We extend the quartic residue symbol to all odd ideals $\mathfrak{n}$ and then to all odd elements $\beta$ in the same way as the quadratic residue symbol. Then we have the following theorem.

\begin{theorem}
\label{t4R}
Let $\alpha, \beta \in \Z[\zeta_8]$ with $\beta$ odd. Then for fixed $\alpha$, the symbol $(\alpha/\beta)_{4, \Q(\zeta_8)}$ depends only on $\beta$ modulo $16\alpha\Z[\zeta_8]$. Furthermore, if $\alpha$ is also odd, we have
\[
\left(\frac{\alpha}{\beta}\right)_{4, \Q(\zeta_8)} = \mu(\alpha, \beta) \left(\frac{\beta}{\alpha}\right)_{4, \Q(\zeta_8)},
\]
where $\mu(\alpha, \beta) \in \{\pm 1, \pm i\}$ depends only on the congruence classes of $\alpha$ and $\beta$ modulo $16$.
\end{theorem}

\begin{proof}
Use Proposition 6.11 of Lemmermeyer \cite[p.\ 199]{Lemmermeyer}.
\end{proof}

\subsection{A fundamental domain}
Let $F$ be a number field of degree $n$ over $\Q$ and let $O_F$ be its ring of integers. Suppose that $F$ has $r$ real embeddings and $s$ pairs of conjugate complex embeddings so that $r + 2s = n$. Define $T$ to be the torsion subgroup of $O_F^\ast$. Then, by Dirichlet's Unit Theorem, there exists a free abelian group $V \subseteq O_F^\ast$ of rank $r + s - 1$ with $O_F^\ast = T \times V$. Fix one choice of such a $V$.

There is a natural action of $V$ on $O_F$. The goal of this subsection is to construct a fundamental domain $\mathcal{D}$ for this action. Such a fundamental domain allows us to transform a sum over ideals into a sum over elements. It will be important that the resulting fundamental domain has nice geometrical properties, so that we have good control over the elements we are summing.

Fix an integral basis $\omega_1, \ldots, \omega_n$ for $O_F$. We view $\omega_1, \ldots, \omega_n$ as an ordered list and write $\omega$ for this ordered list. Then we get an isomorphism of $\Q$-vector spaces $i_\omega: \mathbb{Q}^n \rightarrow F$, where $i_\omega$ is given by $(a_1, \ldots, a_n) \mapsto a_1\omega_1 + \ldots + a_n\omega_n$. For a subset $S \subseteq \mathbb{R}^n$ and an element $\alpha \in F$, we will say that $\alpha \in S$ if $i_{\omega}^{-1}(\alpha) \in S$. Define for our integral basis $\omega$ and a real number $X > 0$
\[
B(X, \omega) := \left\{(x_1, \ldots, x_n) \in \mathbb{R}^n : \left|\prod_{i = 1}^n \left(x_1 \sigma_i(\omega_1) + \ldots + x_n \sigma_i(\omega_n)\right)\right| \leq X\right\},
\]
where $\sigma_1, \ldots, \sigma_n$ are the embeddings of $F$ into $\mathbb{C}$.

\begin{lemma}
\label{lFD}
Let $F$ be a number field with ring of integers $O_F$ and integral basis $\omega = \{\omega_1, \ldots, \omega_n\}$. Choose a splitting $O_F^\ast = T \times V$, where $T$ is the torsion subgroup of $O_F^\ast$. There exists a subset $\mathcal{D} \subseteq \mathbb{R}^n$ such that
\begin{enumerate}
\item[(i)] for all $\alpha \in O_F \setminus \{0\}$, there exists a unique $v \in V$ such that $v \alpha \in \mathcal{D}$. Furthermore, we have the equality
\[
\{u \in O_F^\ast : u\alpha \in \mathcal{D}\} = \{tv : t \in T\};
\]
\item[(ii)] $\mathcal{D} \cap B(1, \omega)$ has an $(n - 1)$-Lipschitz parametrizable boundary;
\item[(iii)] there is a constant $C(\omega)$ depending only on $\omega$ such that for all $\alpha \in \mathcal{D}$ we have $|a_i| \leq C(\omega) \cdot |\Norm(\alpha)|^{\frac{1}{n}}$, where $a_i \in \Z$ are such that $\alpha = a_1 \omega_1 + \ldots + a_n \omega_n$.
\end{enumerate}
\end{lemma}

\begin{proof}
This is Lemma 3.5 of \cite{KM2}.
\end{proof}

We will use Lemma \ref{lFD} for $F := \Q(\zeta_8)$; in order to do so we must make some choices. We choose $V := \langle 1 + \sqrt{2} \rangle$ and integral basis $\omega := \{1, \zeta_8, \zeta_8^2, \zeta_8^3\}$. The resulting fundamental domain will be called $\mathcal{D}$, and we define $\mathcal{D}(X) := \mathcal{D} \cap B(X, \omega)$.

\section{The sieve}
\label{sSieve}
Let $\{a_p\}$ be a sequence of complex numbers indexed by the primes and define
\[
S(X) := \sum_{p \leq X} a_p.
\]
To prove our main theorem, we must prove oscillation of $S(X)$ for the specific sequence $\{e_p\}$ defined in equation (\ref{edep}). There are relatively few methods that can deal with such sums. The most common approach is to attach an $L$-function and then use the zero-free region. This approach requires that our sequence $\{e_p\}$ has good multiplicative properties. It turns out that $\{e_p\}$ is instead twisted multiplicative (see Lemma \ref{lMw1} and Lemma \ref{lMw2}), and this suggests we use Vinogradov's method instead.

Recall that $h(-p)$ denotes the class number of $\text{Cl}(\Q(\sqrt{-p}))$. By definition of $e_p$ we have $e_p = 0$ if and only if $8 \nmid h(-p)$. It is well-known that $\Q(\zeta_8, \sqrt{1 + i})$ is a \emph{governing field} for the $8$-rank of $\text{Cl}(\Q(\sqrt{-p}))$, in fact a prime $p$ splits completely in $\Q(\zeta_8, \sqrt{1 + i})$ if and only if $8 \mid h(-p)$. This is extremely convenient. Indeed, if we apply Vinogradov's method to our governing field, primes of degree $1$ will give the dominant contribution and these primes automatically have $e_p \neq 0$.

Unfortunately, $\Q(\zeta_8, \sqrt{1 + i})$ is a field of degree $8$, which is simply too large to make our analytic methods work unconditionally. Indeed, using the same approach for the sums of type I as \cite{FIMR, FIMRE}, one ends up with short character sums of modulus $q$ and length roughly $q^{\frac{1}{8}}$, which is far outside the reach of the Burgess bound. However, assuming a short character sum conjecture, one still obtains the desired oscillation and this is the approach taken in \cite{KM1}. Instead we work over $\Q(\zeta_8)$; fortunately, $\Q(\zeta_8, \sqrt{1 + i})$ is an abelian extension of $\Q(\zeta_8)$, which implies that the splitting of a prime $\mathfrak{p}$ of $\Q(\zeta_8)$ in the extension $\Q(\zeta_8, \sqrt{1 + i})/\Q(\zeta_8)$ is determined by a congruence condition. Such a congruence condition can easily be incorporated in Vinogradov's method.

We will follow Section 5 of Friedlander et al. \cite{FIMR, FIMRE}, who adapted Vinogradov's method to number fields. Let $F$ be a number field. Define for a non-zero ideal $\mathfrak{n}$ of $O_F$
\[
\Lambda(\mathfrak{n}) := 
\left\{
	\begin{array}{ll}
		\log \Norm\mathfrak{p} & \mbox{if } \mathfrak{n} = \mathfrak{p}^l \\
		0 & \mbox{otherwise} 
	\end{array}
\right.
\]
and suppose that we want to prove oscillation of
\[
S(X) := \sum_{\Norm\mathfrak{n} \leq X} a_{\mathfrak{n}} \Lambda(\mathfrak{n}),
\]
where $a_{\mathfrak{n}}$ is of absolute value at most $1$. The power of Vinogradov's method lies in the fact that one does not have to deal with $S(X)$ directly. Instead one has to prove cancellations of
\[
A(X, \mathfrak{d}) := \sum_{\substack{\Norm{\mathfrak{n}} \leq X \\ \mathfrak{d} \mid \mathfrak{n}}} a_{\mathfrak{n}},
\]
which are traditionally called sums of type I or linear sums, and
\[
B(M, N) := \sum_{\Norm{\mathfrak{m}} \leq M} \sum_{\Norm{\mathfrak{n}} \leq N} \alpha_{\mathfrak{m}} \beta_{\mathfrak{n}} a_{\mathfrak{m}\mathfrak{n}},
\]
which are traditionally called sums of type II or bilinear sums. It is important to remark that $S(X)$ depends only on $a_{\mathfrak{n}}$ with $\mathfrak{n}$ a prime power, while $A(X, \mathfrak{d})$ and $B(M, N)$ certainly do not. This gives a substantial amount of flexibility, since we may define $a_{\mathfrak{n}}$ on composite ideals $\mathfrak{n}$ in any way we like provided that we can prove oscillation of $A(X, \mathfrak{d})$ and $B(M, N)$. Constructing a suitable sequence $a_{\mathfrak{n}}$ will be the goal of Section \ref{sDef}. We are now ready to state the precise version of Vinogradov's method we are going to use.

\begin{prop}
\label{pVin}
Let $F$ be a number field and let $a_{\mathfrak{n}}$ be a sequence of complex numbers, indexed by the ideals of $O_F$, with $|a_{\mathfrak{n}}| \leq 1$. If $0 < \theta_1, \theta_2 < 1$ and $\theta_3 > 0$ are such that
\begin{itemize}
\item we have for all ideals $\mathfrak{d}$ of $O_F$
\begin{align}
\label{eSumI}
A(X, \mathfrak{d}) \ll_{F, a_{\mathfrak{n}}, \theta_1} \frac{X}{\exp\left(\left(\log X\right)^{\theta_1}\right)};
\end{align}
\item we have for all sequences of complex numbers $\{\alpha_{\mathfrak{m}}\}$ and $\{\beta_{\mathfrak{n}}\}$ of absolute value at most $1$
\begin{align}
\label{eSumII}
B(M, N) \ll_{F, a_{\mathfrak{n}}, \theta_2} (M + N)^{\theta_2} (MN)^{1 - \theta_2} (\log MN)^{\theta_3}.
\end{align}
\end{itemize}
Then we have for all $c < \theta_1$
\[
S(X) \ll_{c, F, a_{\mathfrak{n}}, \theta_1, \theta_2, \theta_3} \frac{X}{\exp\left(\left(\log X\right)^c\right)}.
\]
\end{prop}

\begin{proof}
This quickly follows from Proposition 5.1 of \cite{FIMR, FIMRE} with $y := \exp\left(\left(\log X\right)^{\frac{c + \theta_1}{2}}\right)$.
\end{proof}

The remainder of this paper is devoted to the three major tasks that are left. We start by constructing a suitable sequence $a_{\mathfrak{n}}$ in Section \ref{sDef} to which we will apply Proposition \ref{pVin} with $F = \Q(\zeta_8)$. The main result of Section \ref{sSumI} is Proposition \ref{pSumI}, which proves equation (\ref{eSumI}) for $\theta_1 = 0.2$. Finally, we prove in Section \ref{sSumII} that (\ref{eSumII}) holds with $\theta_2 = \frac{1}{24}$; this is the content of Proposition \ref{pSumII}. Once we have proven Proposition \ref{pSumI} and Proposition \ref{pSumII}, the proof of Theorem \ref{tCancel} is complete.

\section{Definition of the sequence}
\label{sDef}
By Gauss genus theory we know that the $2$-part of $\text{Cl}(\Q(\sqrt{-p}))$ is cyclic, and the $2$-part of $\text{Cl}(\Q(\sqrt{-p}))$ is trivial if and only if $p \equiv 3 \bmod 4$. Let us recall a criterion for $16 \mid h(-p)$ due to Leonard and Williams \cite{LW82}. We have
\[
4 \mid h(-p) \Longleftrightarrow p \equiv 1 \bmod 8.
\]
Now suppose that $4 \mid h(-p)$. There exist positive integers $g$ and $h$ satisfying
\[
p = 2g^2 - h^2.
\]
Then a classical result of Hasse \cite{Hasse} is
\[
8 \mid h(-p) \Longleftrightarrow \left(\frac{g}{p}\right) = 1 \text{ and } p \equiv 1 \bmod 8
\]
or equivalently
\[
8 \mid h(-p) \Longleftrightarrow \left(\frac{-1}{g}\right) = 1 \text{ and } p \equiv 1 \bmod 8.
\]
We are now ready to state the result of Leonard and Williams \cite{LW82}. If $p$ is a prime number with $8 \mid h(-p)$, we have
\[
16 \mid h(-p) \Longleftrightarrow \left(\frac{g}{p}\right)_4 \left(\frac{2h}{g}\right) = 1.
\]
With this in mind, we are going to define a sequence $\{a_\mathfrak{n}\}$, indexed by the integral ideals of $\Z[\zeta_8]$, such that for all unramified prime ideals $\mathfrak{p}$ in $\Z[\zeta_8]$ of norm $p$
\begin{align}
\label{eap}
a_\mathfrak{p} = 
\left\{
	\begin{array}{ll}
		1  & \mbox{if } 16 \mid h(-p) \\
		-1 & \mbox{if } 8 \mid h(-p), 16 \nmid h(-p) \\
		0 & \mbox{otherwise.}
	\end{array}
\right.
\end{align}
The sequence $\{a_\mathfrak{n}\}$ will be constructed in such a way that we can prove the two estimates in Proposition \ref{pSumI} and Proposition \ref{pSumII}. Before we move on, it will be useful to recall some standard facts about $\Z[\zeta_8]$. The ring $\Z[\zeta_8]$ is a PID with unit group generated by $\zeta_8$ and $\epsilon := 1 + \sqrt{2}$. Odd primes are unramified in $\Z[\zeta_8]$, while $2$ is totally ramified. Furthermore, an odd prime $p$ splits completely in $\Z[\zeta_8]$ if and only if $p \equiv 1 \bmod 8$ if and only if $4 \mid h(-p)$. We will make extensive use of the following field diagram.

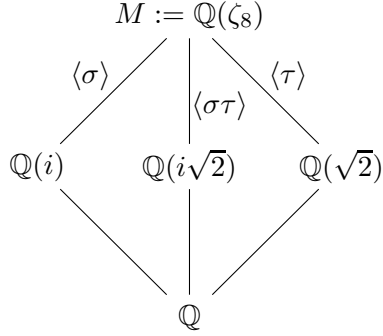
\begin{center}
\begin{tikzpicture}
  \draw (0, 0) node[]{$\Q$};
  \draw (0, 2) node[]{$\Q(i\sqrt{2})$};
	\draw (-2, 2) node[]{$\Q(i)$};
	\draw (2, 2) node[]{$\Q(\sqrt{2})$};
  \draw (0, 4) node[]{$M := \Q(\zeta_8)$};
	\draw (0, 0.3) -- (0, 1.7);
	\draw (0, 2.3) -- (0, 3.7);
	\draw (0.3, 0.3) -- (1.7, 1.7);
	\draw (-0.3, 0.3) -- (-1.7, 1.7);
	\draw (1.7, 2.3) -- (0.3, 3.7);
	\draw (-1.7, 2.3) -- (-0.3, 3.7);
	\draw (1.3, 3.2) node[]{$\left\langle \tau\right\rangle$};
	\draw (0.4, 2.8) node[]{$\left\langle \sigma\tau \right\rangle$};
	\draw (-1.3, 3.2) node[]{$\left\langle \sigma\right\rangle$};
\end{tikzpicture}
\end{center} 

\noindent If $\mathfrak{n}$ is not odd, we set $a_\mathfrak{n} := 0$. From now on $\mathfrak{n}$ is an odd, integral, non-zero ideal of $\Z[\zeta_8]$ and $w$ is a generator of $\mathfrak{n}$. We can write $w$ as
\[
w = a + b\zeta_8 + c\zeta_8^2 + d\zeta_8^3
\]
for certain $a, b, c, d \in \Z$. Define $u, v \in \Z$ by
\[
w \tau(w) = u + v\sqrt{2}.
\]
We can explicitly compute $u$ and $v$ using the following formulas
\begin{align}
\label{eFormulau}
u = \frac{w\tau(w) + \sigma(w)\sigma\tau(w)}{2} = a^2 + b^2 + c^2 + d^2
\end{align}
and
\begin{align}
\label{eFormulav}
v = \frac{w\tau(w) - \sigma(w)\sigma\tau(w)}{2\sqrt{2}} = ab - ad + bc + cd.
\end{align}
Since $w$ is odd, it follows that $\Norm w \equiv 1 \bmod 8$. Then it follows from
\[
\Norm w = u^2 - 2v^2
\]
that $u$ is an odd integer and $v$ is an even integer. Set 
\[
g := u + v, \quad h := u + 2v,
\]
so that $g$ is an odd integer and $h$ is an odd integer, not necessarily positive. We claim that $g$ is positive. Indeed
\begin{align*}
g &= a^2 + b^2 + c^2 + d^2 + ab - ad + bc + cd \\
&= \frac{1}{2}(a + b)^2 + \frac{1}{2}(a - d)^2 + \frac{1}{2}(b + c)^2 + \frac{1}{2}(c + d)^2 > 0.
\end{align*}
By construction $g$ and $h$ satisfy
\[
\Norm w = 2g^2 - h^2.
\]
We start by showing that the value of
\begin{align}
\label{e8p}
\left(\frac{-1}{g}\right)
\end{align}
does not depend on the choice of generator $w$ of our ideal $\mathfrak{n}$.

\begin{lemma}
\label{l8p}
Let $\mathfrak{n}$ be an odd, integral ideal of $\Z[\zeta_8]$. Then the value of equation (\ref{e8p}) is the same for all generators $w$ of $\mathfrak{n}$.
\end{lemma}

\begin{proof}
Suppose that we replace $w$ by $\zeta_8 w$. Because $\zeta_8 \tau(\zeta_8) = 1$, it follows that $u$ and $v$, hence also $g$, do not change. Suppose instead that we replace $w$ by $\epsilon w$. In this case $u$ becomes $3u + 4v$ and $v$ becomes $2u + 3v$, so $g$ becomes $5u + 7v$. Hence our lemma boils down to
\[
\left(\frac{-1}{u + v}\right) = \left(\frac{-1}{5u + 7v}\right),
\]
which holds if and only if
\[
u + v \equiv 5u + 7v \bmod 4.
\]
But recall that $v$ is even by our assumption that $w$ is odd.
\end{proof}

We define for odd $w \in \Z[\zeta_8]$ the following symbol
\[
[w] := \left(\frac{g}{w}\right)_{4, M} \left(\frac{2h}{g}\right),
\]
where we remind the reader that $M$ is defined to be $\Q(\zeta_8)$. We express this as
\begin{align}
\label{ewsplit}
[w] = [w]_1[w]_2, \quad [w]_1 := \left(\frac{g}{w}\right)_{4, M}, \quad [w]_2 := \left(\frac{2h}{g}\right).
\end{align}
It is easily checked that $[\zeta_8 w] = [w]$. Unfortunately, it is not always true that $[\epsilon w] = [w]$. To get around this, we need the following lemma.

\begin{lemma}
\label{lUnitAction}
We have for all odd $w$
\[
[\epsilon^4 w] = [w].
\]
\end{lemma}

\begin{proof}
We have for any odd $w$
\begin{align}
\label{ew1}
[w]_1 = \left(\frac{g}{w}\right)_{4, M} = \left(\frac{u + v}{w}\right)_{4, M} = \left(\frac{\left(\frac{1}{2} - \frac{1}{2\sqrt{2}}\right)\sigma(w)\sigma\tau(w)}{w}\right)_{4, M},
\end{align}
where we use the explicit formulas for $u$ and $v$, see equation (\ref{eFormulau}) and equation (\ref{eFormulav}), in terms of $w$. From this expression it quickly follows that $[\epsilon^2 w]_1 = [w]_1$. We also have
\begin{align}
\label{ew2uv}
[w]_2 &= \left(\frac{2h}{g}\right) = \left(\frac{2u + 4v}{u + v}\right) = \left(\frac{2}{u + v}\right) \left(\frac{v}{u + v}\right) \nonumber \\
&= \left(\frac{2}{u + v}\right) \left(\frac{-u}{u + v}\right) = \left(\frac{-2}{u + v}\right) \left(\frac{v}{u}\right) (-1)^{\frac{u - 1}{2} \cdot \frac{u + v - 1}{2}}.
\end{align}
A straightforward computation shows that the $u$ and $v$ associated to $\epsilon^4 w$ are respectively $u_1 := 577u + 816v$ and $v_1 := 408u + 577v$. Then we have
\begin{align}
\label{evu1}
\left(\frac{v}{u}\right) = \left(\frac{408u + 577v}{577u + 816v}\right) = \left(\frac{v_1}{u_1}\right)
\end{align}
due to Proposition 2 in Milovic \cite{Milovic2}. It will be useful to observe that the following congruences hold true
\[
u \equiv u_1 \bmod 8, \quad v \equiv v_1 \bmod 8.
\]
This immediately implies
\begin{align}
\label{evu2}
\left(\frac{-2}{u + v}\right) = \left(\frac{-2}{u_1 + v_1}\right),
\end{align}
and therefore the lemma.
\end{proof}

With this out of the way, we have all the tools necessary to define $a_{\mathfrak{n}}$. Suppose that $\mathfrak{n}$ is an odd, integral ideal of $\Z[\zeta_8]$ with generator $w$. Then we define
\begin{align}
\label{ean}
a_{\mathfrak{n}} := 
\left\{
	\begin{array}{ll}
		\frac{1}{4}\left([w] + [\epsilon w] + [\epsilon^2 w] + [\epsilon^3 w]\right)  & \mbox{if } w \text{ satisfies } (\ref{e8p}) \\
		0 & \mbox{otherwise.}
	\end{array}
\right.
\end{align}
for any generator $w$ of $\mathfrak{n}$. Here we say that $w$ satisfies equation (\ref{e8p}) if $(-1/g) = 1$, where $g$ is defined in terms of $w$ as above. Then an application of Lemma \ref{l8p} and Lemma \ref{lUnitAction} shows that (\ref{ean}) is indeed well-defined.

\begin{lemma}
\label{lanagrees}
The sequence $a_{\mathfrak{n}}$ satisfies equation (\ref{eap}) for all unramified prime ideals $\mathfrak{p}$ of degree $1$ in $\Z[\zeta_8]$.
\end{lemma}

\begin{proof}
Let $\mathfrak{p}$ be an unramified prime ideal of degree $1$ in $\Z[\zeta_8]$ and let $w$ be a generator of $\mathfrak{p}$. Put $p := \Norm w$. Lemma \ref{l8p} and the aforementioned result of Hasse imply
\[
w \text{ does not satisfy } (\ref{e8p}) \Longleftrightarrow 8 \nmid h(-p),
\]
and $a_{\mathfrak{p}}$ is indeed $0$ in this case. Now suppose that $w$ does satisfy (\ref{e8p}). Recall that
\[
[w] = \left(\frac{g}{w}\right)_{4, M} \left(\frac{2h}{g}\right),
\]
where $g$ and $h$ are explicit functions of $w$. We stress that these $g$ and $h$ are not necessarily the same $g$ and $h$ from Leonard and Williams. Indeed, Leonard and Williams require $g$ and $h$ to be positive, while our $h$ is not necessarily positive. However, since $w$ satisfies (\ref{e8p}), their criterion remains valid irrespective of the sign of $h$. Then, the criterion implies
\[
[w] = [\epsilon w] = [\epsilon^2 w] = [\epsilon^3 w].
\]
Furthermore, the criterion also shows that
\[
[w] = 1 \Longleftrightarrow 16 \mid h(-p).
\]
This completes the proof of our lemma.
\end{proof}

\section{Sums of type I}
\label{sSumI}
The goal of this section is to bound the following sum
\[
A(X, \mathfrak{d}) = \sum_{\substack{\Norm{\mathfrak{n}} \leq X \\ \mathfrak{d} \mid \mathfrak{n}}} a_{\mathfrak{n}} = \sum_{\substack{\Norm{\mathfrak{n}} \leq X \\ \mathfrak{d} \mid \mathfrak{n}, \mathfrak{n} \text{ odd}}} a_{\mathfrak{n}}.
\]
By picking a generator for $\mathfrak{n}$ we obtain
\[
A(X, \mathfrak{d}) = \frac{1}{8} \sum_{\substack{w \in \mathcal{D}(X) \\ w \equiv 0 \bmod \mathfrak{d} \\ w \text{ odd}}} a_{(w)} = \frac{1}{32} \sum_{\substack{w \in \mathcal{D}(X) \\ w \equiv 0 \bmod \mathfrak{d} \\ w \text{ odd}}} \mathbf{1}_{w \text{ sat. } (\ref{e8p})} \left([w] + [\epsilon w] + [\epsilon^2 w] + [\epsilon^3 w]\right).
\]
We define for $i = 0, \ldots, 3$ and $\rho$ an invertible congruence class modulo $2^{10}$
\[
A(X, \mathfrak{d}, u_i, \rho) := \sum_{\substack{w \in u_i\mathcal{D}(X) \\ w \equiv 0 \bmod \mathfrak{d} \\ w \equiv \rho \bmod 2^{10}}} [w] = \sum_{\substack{w \in u_i\mathcal{D}(X) \\ w \equiv 0 \bmod \mathfrak{d} \\ w \equiv \rho \bmod 2^{10}}} \left(\frac{g}{w}\right)_{4, M} \left(\frac{2h}{g}\right),
\]
where $u_i := \epsilon^i$. With this definition in place, we may split $A(X, \mathfrak{d})$ as follows
\[
A(X, \mathfrak{d}) = \frac{1}{32} \sum_{i = 0}^3 \sum_{\rho \in (O_M/2^{10}O_M)^\ast} \mathbf{1}_{\rho \text{ sat. } (\ref{e8p})} A(X, \mathfrak{d}, u_i, \rho),
\]
since the truth of equation (\ref{e8p}) depends only on $w$ modulo $4$. Then it is enough to bound each individual sum $A(X, \mathfrak{d}, u_i, \rho)$. In order to bound this sum, our first step is to carefully rewrite the symbol $[w]$ in a more tractable form. While doing so, we will find some hidden cancellation between $[w]_1$ and $[w]_2$ that is vital for making our results unconditional. 

Throughout this section we use the convention that $\mu(\cdot) \in \{\pm 1, \pm i\}$ is a function depending only on the variables between the parentheses; at each occurence $\mu(\cdot)$ may be a different function. Since our cancellation will come from fixing $b$, $c$ and $d$ while varying $a$, factors of the shape $\mu(\rho, b, c, d)$ will present no issues for us. Let us start by rewriting $[w]_2$. It follows from equation (\ref{ew2uv}) that
\begin{align}
\label{ew1Re1}
\left(\frac{2h}{g}\right) = \left(\frac{v}{u}\right) \mu(\rho).
\end{align}
Using the formulas for $u$ and $v$ we get
\begin{align}
\label{ew1Re2}
\left(\frac{v}{u}\right) = \left(\frac{ab - ad + bc + cd}{a^2 + b^2 + c^2 + d^2}\right).
\end{align}
If $v$ is not zero, we can uniquely factor $v$ as
\begin{align}
\label{efactorv}
v := v_1v_2t,
\end{align}
where $v_1$ is an odd, positive integer satisfying $\gcd(v_1, b - d) = 1$, $v_2$ is an odd integer consisting only of primes dividing $b - d$ and $t$ is positive and only divisible by powers of $2$. Then we have
\begin{align}
\label{ew1Re3}
\left(\frac{ab - ad + bc + cd}{a^2 + b^2 + c^2 + d^2}\right) = \left(\frac{v_1}{a^2 + b^2 + c^2 + d^2}\right) \left(\frac{tv_2}{a^2 + b^2 + c^2 + d^2}\right).
\end{align}
Let $\rho'$ be the congruence class of $v_1$ modulo $8$. Using the following identity modulo $v$
\[
a^2(b - d)^2 \equiv c^2(b + d)^2 \bmod v
\]
and the fact that this identity continues to hold for any divisor of $v$, so in particular for $v_1$, we rewrite the first factor of equation (\ref{ew1Re3}) as follows
\begin{align}
\label{ew1Re4}
\left(\frac{v_1}{a^2 + b^2 + c^2 + d^2}\right) &= \mu(\rho, \rho') \left(\frac{a^2 + b^2 + c^2 + d^2}{v_1}\right) \nonumber \\
&= \mu(\rho, \rho') \left(\frac{(a^2 + b^2 + c^2 + d^2)(b - d)^2}{v_1}\right) \nonumber \\
&= \mu(\rho, \rho') \left(\frac{a^2(b - d)^2 + (b^2 + c^2 + d^2)(b - d)^2}{v_1}\right) \nonumber \\
&= \mu(\rho, \rho') \left(\frac{c^2(b + d)^2 + (b^2 + c^2 + d^2)(b - d)^2}{v_1}\right) \nonumber \\
&= \mu(\rho, \rho') \left(\frac{(b^2 + d^2)(2c^2 + (b - d)^2)}{v_1}\right).
\end{align}
Stringing together (\ref{ew1Re1}), (\ref{ew1Re2}), (\ref{ew1Re3}) and (\ref{ew1Re4}), we conclude that
\begin{align}
\label{eConc1}
\left(\frac{2h}{g}\right) = \mu(\rho, \rho') \left(\frac{(b^2 + d^2)(2c^2 + (b - d)^2)}{v_1}\right) \left(\frac{tv_2}{a^2 + b^2 + c^2 + d^2}\right).
\end{align}
Our next goal is to simplify $[w]_1$. We have by equation (\ref{ew1}) and Theorem \ref{t4R}
\begin{align}
\label{ew2Re1}
\left(\frac{g}{w}\right)_{4, M} = \left(\frac{\left(\frac{1}{2} - \frac{1}{2\sqrt{2}}\right)\sigma(w)\sigma\tau(w)}{w}\right)_{4, M} = \mu(\rho) \left(\frac{\sigma(w)\sigma\tau(w)}{w}\right)_{4, M}.
\end{align}
The quartic residue symbol in equation (\ref{ew2Re1}) is the product of two quartic residue symbols. One of them is equal to
\begin{align}
\label{ew2Re2}
\left(\frac{\sigma\tau(w)}{w}\right)_{4, M} &= \left(\frac{a + d\zeta_8 - c\zeta_8^2 + b\zeta_8^3}{a + b\zeta_8 + c\zeta_8^2 + d\zeta_8^3}\right)_{4, M} = \left(\frac{-2c\zeta_8^2 + (d - b)(\zeta_8 - \zeta_8^3)}{a + b\zeta_8 + c\zeta_8^2 + d\zeta_8^3}\right)_{4, M} \nonumber \\
&= \left(\frac{\zeta_8^2}{a + b\zeta_8 + c\zeta_8^2 + d\zeta_8^3}\right)_{4, M} \left(\frac{-2c + (b - d)(\zeta_8 + \zeta_8^3)}{a + b\zeta_8 + c\zeta_8^2 + d\zeta_8^3}\right)_{4, M} \nonumber \\
&= \mu(\rho) \left(\frac{-2c + (b - d)(\zeta_8 + \zeta_8^3)}{a + b\zeta_8 + c\zeta_8^2 + d\zeta_8^3}\right)_{4, M},
\end{align}
where the last equality is due to Theorem \ref{t4R}. For the remainder of this section we assume that $b - d$ is not zero. We factor $-2c + (b - d)(\zeta_8 + \zeta_8^3)$ in the ring $\Z[\sqrt{-2}]$ as
\[
-2c + (b - d)(\zeta_8 + \zeta_8^3) = \eta^4 e_0 e
\]
with $\eta$ and $e_0$ consisting only of even prime factors, $e_0$ not divisible by a non-trivial fourth power and $e$ odd. This factorization is unique up to multiplication by units. Then we have by Theorem \ref{t4R}
\begin{align}
\label{ew2Re3}
\left(\frac{-2c + (b - d)(\zeta_8 + \zeta_8^3)}{a + b\zeta_8 + c\zeta_8^2 + d\zeta_8^3}\right)_{4, M} = \mu(\rho, b, c, d) \left(\frac{a + b\zeta_8 + c\zeta_8^2 + d\zeta_8^3}{e}\right)_{4, M}.
\end{align}
But a simple computation shows
\[
a + b\zeta_8 + c\zeta_8^2 + d\zeta_8^3 \equiv \sigma\tau(a + b\zeta_8 + c\zeta_8^2 + d\zeta_8^3) \bmod e.
\]
Let $\mathfrak{p}$ be a prime in $\Z[\sqrt{-2}]$ that divides $e$. Then we may replace $a + b\zeta_8 + c\zeta_8^2 + d\zeta_8^3$ by some element in $\Z[\sqrt{-2}]$ by Lemma 3.4 of \cite{KM2}. In case $\mathfrak{p}$ splits in $M$, we apply Lemma 3.2 of \cite{KM2}. While if $\mathfrak{p}$ remains inert, we see that $\mathfrak{p}$ is of degree $1$ and $\Norm \mathfrak{p} \equiv 3 \bmod 8$. In this case we apply Lemma 3.3 of \cite{KM2}. Hence in all cases
\[
\left(\frac{a + b\zeta_8 + c\zeta_8^2 + d\zeta_8^3}{\mathfrak{p}}\right)_{4, M} = \mathbbm{1}_{\gcd(a + b\zeta_8 + c\zeta_8^2 + d\zeta_8^3, \mathfrak{p}) = (1)}.
\]
This yields
\begin{align}
\label{ew2Re4}
\left(\frac{a + b\zeta_8 + c\zeta_8^2 + d\zeta_8^3}{e}\right)_{4, M} = \mathbbm{1}_{\gcd(w, \sigma\tau(w)) = (1)}.
\end{align}
We deduce from equation (\ref{ew2Re2}), (\ref{ew2Re3}) and (\ref{ew2Re4}) that
\begin{align}
\label{eConc2}
\left(\frac{\sigma\tau(w)}{w}\right)_{4, M} = \mu(\rho, b, c, d) \mathbbm{1}_{\gcd(w, \sigma\tau(w)) = (1)}.
\end{align}
We will now study the other quartic residue symbol in equation (\ref{ew2Re1}) using very similar methods. We start with the identity
\begin{align}
\label{ew2Re5}
\left(\frac{\sigma(w)}{w}\right)_{4, M} &= \left(\frac{a - b\zeta_8 + c\zeta_8^2 - d\zeta_8^3}{a + b\zeta_8 + c\zeta_8^2 + d\zeta_8^3}\right)_{4, M} = \left(\frac{-2\zeta_8(b + d\zeta_8^2)}{a + b\zeta_8 + c\zeta_8^2 + d\zeta_8^3}\right)_{4, M} \nonumber \\
&= \left(\frac{-2\zeta_8}{a + b\zeta_8 + c\zeta_8^2 + d\zeta_8^3}\right)_{4, M} \left(\frac{b + d\zeta_8^2}{a + b\zeta_8 + c\zeta_8^2 + d\zeta_8^3}\right)_{4, M} \nonumber \\
&= \mu(\rho) \left(\frac{b + d\zeta_8^2}{a + b\zeta_8 + c\zeta_8^2 + d\zeta_8^3}\right)_{4, M},
\end{align}
where we use Theorem \ref{t4R} once more. We choose $i := \zeta_8^2$ and factor $b + di$ in the ring $\Z[i]$ as
\[
b + di = \eta'^4 e'_0 e'
\]
with $\eta'$ and $e'_0$ consisting only of even prime factors, $e'_0$ not divisible by a non-trivial fourth power and $e'$ odd. Such a factorization is unique up to multiplication by units. With this factorization we have due to Theorem \ref{t4R}
\begin{align}
\label{ew2Re6}
\left(\frac{b + di}{a + b\zeta_8 + c\zeta_8^2 + d\zeta_8^3}\right)_{4, M} = \mu(\rho, b, c, d) \left(\frac{a + b\zeta_8 + c\zeta_8^2 + d\zeta_8^3}{e'}\right)_{4, M}.
\end{align}
We claim that
\begin{align}
\label{ew2Re7}
\left(\frac{a + b\zeta_8 + c\zeta_8^2 + d\zeta_8^3}{e'}\right)_{4, M} = \left(\frac{a + c\zeta_8^2}{e'}\right)_{4, M} = \left(\frac{a + ci}{e'}\right)_{2, \Q(i)}.
\end{align}
Indeed, let $\mathfrak{p}$ be a prime in $\Z[i]$ that divides $e'$. If $\mathfrak{p}$ splits in $M$, Lemma 3.2 of \cite{KM2} shows that
\[
\left(\frac{a + c\zeta_8^2}{\mathfrak{p}}\right)_{4, M} = \left(\frac{a + ci}{\mathfrak{p}}\right)_{2, \Q(i)}.
\]
Suppose instead that $\mathfrak{p}$ remains inert. Then $\mathfrak{p}$ is of degree $1$ and $\Norm \mathfrak{p} \equiv 5 \bmod 8$. Now we apply Lemma 3.3 of \cite{KM2} to obtain
\[
\left(\frac{a + c\zeta_8^2}{\mathfrak{p}}\right)_{4, M} = \left(\frac{a + ci}{\mathfrak{p}}\right)_{2, \Q(i)}.
\]
This establishes our claim and hence equation (\ref{ew2Re6}). Combining (\ref{ew2Re5}), (\ref{ew2Re6}) and (\ref{ew2Re7}) acquires the validity of
\begin{align}
\label{eConc3}
\left(\frac{\sigma(w)}{w}\right)_{4, M} = \mu(\rho, b, c, d) \left(\frac{a + ci}{e'}\right)_{2, \Q(i)}.
\end{align}
Put
\[
f(w, \rho) := \mu(\rho, \rho', b, c, d) \mathbbm{1}_{\gcd(w, \sigma\tau(w)) = (1)} \left(\frac{tv_2}{a^2 + b^2 + c^2 + d^2}\right).
\]
Using (\ref{eConc1}), (\ref{eConc2}) and (\ref{eConc3}), we conclude that
\begin{align}
\label{efwr}
\left(\frac{g}{w}\right)_{4, M} \left(\frac{2h}{g}\right) = f(w, \rho) \left(\frac{(b^2 + d^2)(2c^2 + (b - d)^2)}{v_1}\right) \left(\frac{a + ci}{e'}\right)_{2, \Q(i)}.
\end{align}
Our hidden cancellation will come from comparing the Jacobi symbols
\[
\left(\frac{b^2 + d^2}{v_1}\right) \text{  and  } \left(\frac{a + ci}{e'}\right)_{2, \Q(i)}.
\]
Our goal is to show that these two Jacobi symbols are equal up to some easily controlled factors. We can uniquely factor 
\[
b^2 + d^2 = z_1z_2,
\]
where $z_1$ and $z_2$ are positive integers satisfying 
\begin{itemize}
\item $(z_1, z_2) = 1$;
\item $z_1$ odd and squarefree;
\item if $p$ is odd and divides $z_2$, then also $p^2$ divides $z_2$.
\end{itemize}
With this factorization we have
\[
\left(\frac{b^2 + d^2}{v_1}\right) = \left(\frac{z_1}{v_1}\right) \left(\frac{z_2}{v_1}\right) = \mu(\rho', b, c, d) \left(\frac{v_1}{z_1}\right) \left(\frac{z_2}{v_1}\right).
\]
In a similar vein we uniquely factor, up to multiplication by units, $e'$ in $\Z[i]$ as
\[
e' = \gamma_1 \gamma_2
\]
with $(\Norm\gamma_1, \Norm\gamma_2) = (1)$, $\Norm\gamma_1$ squarefree and $\Norm\gamma_2$ squarefull. The point of this factorization is that $\Norm\gamma_1 = z_1$. This gives
\[
\left(\frac{v_1}{z_1}\right) = \left(\frac{v_1}{\gamma_1}\right)_{2, \Q(i)}.
\]
We claim that
\begin{align}
\label{eGcd}
(tv_2, \gamma_1) = (d, \gamma_1) = (1).
\end{align}
We clearly have $(t, \gamma_1) = (1)$, so we first show that $(v_2, \gamma_1) = (1)$. Let $\mathfrak{p}$ be an odd prime of $\Z[i]$ above $p$ such that $\mathfrak{p} \mid v_2$ and $\mathfrak{p} \mid \gamma_1$. Then we have $p \mid v_2$ and $\Norm\mathfrak{p} \mid \Norm\gamma_1$. However, $v_2$ is composed entirely of primes dividing $b - d$, while $\Norm\gamma_1$ divides $b^2 + d^2$. We conclude that $p$ divides both $b$ and $d$. But then $p$ can not divide $\gamma_1$ by construction. We can prove in a similar way that $(d, \gamma_1) = (1)$, thus proving the claim.

From equation (\ref{eGcd}) we acquire the validity of
\begin{align*}
\left(\frac{v_1}{z_1}\right) &= \left(\frac{v_1}{\gamma_1}\right)_{2, \Q(i)} = \mu(b, c, d, t) \left(\frac{v_2}{\gamma_1}\right)_{2, \Q(i)} \left(\frac{v}{\gamma_1}\right)_{2, \Q(i)} \\
&= \mu(b, c, d, t) \left(\frac{v_2}{\gamma_1}\right)_{2, \Q(i)} \left(\frac{a + ci}{\gamma_1}\right)_{2, \Q(i)} \left(\frac{-d(1 + i)}{\gamma_1}\right)_{2, \Q(i)} \\
&= \mu(b, c, d, t) \left(\frac{v_2}{\gamma_1}\right)_{2, \Q(i)} \left(\frac{a + ci}{\gamma_1}\right)_{2, \Q(i)},
\end{align*}
where we use the identity
\[
v = ab - ad + bc + cd \equiv -ad(1 + i) + cd(1 - i) = -d(1 + i)(a + ci) \bmod \gamma_1.
\]
We conclude that
\begin{multline}
\label{egwr}
\left(\frac{b^2 + d^2}{v_1}\right) \left(\frac{a + ci}{e'}\right)_{2, \Q(i)} = \\
\mu(\rho, \rho', b, c, d, t) \left(\frac{z_2}{v_1}\right) \left(\frac{v_2}{\gamma_1}\right)_{2, \Q(i)} \left(\frac{a + ci}{\gamma_2}\right)_{2, \Q(i)} \mathbbm{1}_{\gcd(a + ci, \gamma_1) = (1)}.
\end{multline}
Put
\begin{multline}
g(w, \rho) := \mu(\rho, \rho', b, c, d, t) \left(\frac{tv_2}{a^2 + b^2 + c^2 + d^2}\right)\\
\left(\frac{z_2}{v_1}\right) \left(\frac{v_2}{\gamma_1}\right)_{2, \Q(i)} \left(\frac{a + ci}{\gamma_2}\right)_{2, \Q(i)} \mathbbm{1}_{\gcd(a + ci, \gamma_1) = \gcd(w, \sigma\tau(w)) = (1)}. \nonumber
\end{multline}
After combining equations (\ref{efwr}) and (\ref{egwr}), we get
\begin{align*}
\left(\frac{g}{w}\right)_{4, M} \left(\frac{2h}{g}\right) &= g(w, \rho) \left(\frac{2c^2 + (b - d)^2}{v_1}\right) \\
&= \mu(\rho, \rho', b, c, d, t) g(w, \rho) \left(\frac{v_1}{2c^2 + (b - d)^2}\right).
\end{align*}
With this formula we have finally rewritten our symbol in a satisfactory manner; we now return to estimating the sum $A(X, \mathfrak{d}, u_i, \rho)$. We recall the factorization $v = v_1v_2t$, where $v_1$ is an odd, positive integer satisfying $\gcd(v_1, b - d) = 1$, $v_2$ is an odd integer consisting only of primes dividing $b - d$ and $t$ is positive and only divisible by powers of $2$. We further recall that $\rho'$ is the congruence class of $v_1$ modulo $8$.

Let $2^\alpha$ be the closest integer power of $2$ to $X^{\frac{1}{100}}$. We fix $b, c, d$ such that $b - d$ has $2$-adic valuation at most $\frac{\alpha}{2}$. If $a$ modulo $2^\alpha$ is given, we claim that $v_{\text{odd}}$ is determined modulo $8$, where $v_{\text{odd}}$ is the odd part of
\begin{align}
\label{vOdd}
v = a(b - d) + c(b + d),
\end{align}
with the exception of $\ll X^{\frac{1}{200}}$ congruence classes $\rho''$ for $a$ modulo $2^\alpha$. Note that, for fixed $b$, $c$ and $d$, $\rho''$ determines $v$ modulo $2^\alpha$. If $\alpha \geq 3$, $v$ modulo $2^\alpha$ determines $v_{\text{odd}}$ modulo 8 unless $v$ is divisible by $2^{\alpha - 3}$. There are only $8$ congruence classes modulo $2^\alpha$ divisible by $2^{\alpha - 3}$. Now take such a congruence class, say $\rho'''$. But there are $\ll X^{\frac{1}{200}}$ congruence classes $\rho''$ modulo $2^\alpha$ with
\[
\rho''(b - d) + c(b + d) \equiv \rho''' \bmod 2^\alpha
\]
by our assumption that the $2$-adic valuation of $b - d$ is at most $\frac{\alpha}{2}$, and our claim follows.

Similarly, we know the value of $t$ with the exception of $\ll X^{\frac{1}{200}}$ congruence classes for $a$ modulo $2^\alpha$. We remove all such congruence classes from the sum, which gives an error of size at most $X^{\frac{199}{200}}$. From now on we assume that $a$ does not lie in such a congruence class. For the remaining congruence classes modulo $2^\alpha$, we observe that $\rho'$ is determined by $v_{\text{odd}}$ modulo $8$ together with $b$, $c$ and $d$. Hence both $\rho'$ and $t$ are determined by $a$ modulo $2^\alpha$. 

We would also like to treat $v_2$ as fixed, and we use a similar technique to achieve this. Once more we fix $b$, $c$ and $d$. We assume that
\[
\gcd(b - d, bc + cd) \leq \exp\left(\left(\log X\right)^{0.25}\right).
\]
We can uniquely factor a positive integer $n$ as $x_1x_2$, where $\gcd(x_1, x_2) = 1$, $x_1 > 0$ is squarefree and $x_2 > 0$ is squarefull. We call $x_1$ the squarefree part, and $x_2$ the squarefull part. We further assume that the squarefull part of $b - d$ is of size at most $\exp\left(\left(\log X\right)^{0.25}\right)$. We now factor
\[
\gcd(b - d, bc + cd) = \prod_{i = 1}^k p_i^{f_i}.
\]
Define $f_i'(p_i)$ to be the smallest integer such that
\[
p_i^{f_i'(p_i)} \geq \exp\left(2\left(\log X\right)^{0.25}\right)
\]
and define
\[
G := \prod_{i = 1}^k p_i^{f_i'(p_i)}.
\]
Clearly, we have that $\gcd(b - d, bc + cd)$ divides $G$, since the squarefull part of $b - d$ is of size at most $\exp\left(\left(\log X\right)^{0.25}\right)$. If $a$ modulo $G$ is given, we claim that $v_2$ is determined modulo $G$ with the exception of at most 
\[
\ll \log X \min_{1 \leq i \leq k} \frac{G}{p_i^{f_i'(p_i)}}
\]
congruence classes $\rho''$ for $a$ modulo $G$. Take a prime divisor $p_i$ of $b - d$. If $p_i$ does not divide $bc + cd$, then clearly
\[
p_i \nmid a(b - d) + bc + cd,
\]
so we have found the $p_i$ valuation of $a(b - d) + bc + cd$. Now suppose that $p_i$ also divides $bc + cd$. Then we know the $p_i$ valuation unless
\[
a(b - d) + bc + cd \equiv 0 \bmod p_i^{f_i'(p_i)}.
\]
However, we know that the $p_i$ valuation of $b - d$ is at most $f_i'(p_i)/2$. Hence there are at most $p_i^{f_i'(p_i)/2}$ congruence classes for $a$ modulo $p_i^{f_i'(p_i)/2}$ for which
\[
a(b - d) + bc + cd \equiv 0 \bmod p_i^{f_i'(p_i)},
\]
and we call such a congruence class forbidden. We let $G_i$ be the set of forbidden congruence classes modulo $p_i^{f_i(p_i)}$. Now we discard all congruence classes $\rho''$ modulo $G$ for which there exists a prime $p_i$ dividing $\gcd(b - d, bc + cd)$ such that the reduction of $\rho''$ modulo $p_i^{f_i(p_i)}$ lies in $G_i$. This proves the claim.

Set
\begin{align}
\label{eDefm}
m := \text{lcm}\left(G, z_2, \Norm \gamma_2, 2^\alpha, 2^{10}\right).
\end{align}
Then
\[
\left(\frac{tv_2}{a^2 + b^2 + c^2 + d^2}\right) \left(\frac{z_2}{v_1}\right) \left(\frac{v_2}{\gamma_1}\right)_{2, \Q(i)} \left(\frac{a + ci}{\gamma_2}\right)_{2, \Q(i)}
\]
depends only on $a$ modulo $m$, $b$, $c$ and $d$. If we write $\beta := b\zeta_8 + c\zeta_8^2 + d\zeta_8^3$, we have the following estimate
\[
A(X, \mathfrak{d}, u_i, \rho) \ll \sum_\beta \sum_{f \in \Z/m\Z} \left|\sum_{\substack{a \in \Z \\ a \text{ sat. } (\ast)}} \hspace{-0.3cm} \left(\frac{v_1}{2c^2 + (b - d)^2}\right) \mathbbm{1}_{\gcd(a + ci, \gamma_1) = \gcd(a + \beta, \sigma\tau(a + \beta)) = (1)}\right|,
\]
where $(\ast)$ are the simultaneous conditions
\[
a + \beta \in u_i\mathcal{D}(X), \quad a + \beta \equiv 0 \bmod \mathfrak{d}, \quad a + \beta \equiv \rho \bmod 2^{10}, \quad a \equiv f \bmod m.
\]
Recall that the condition $a + \beta \in u_i\mathcal{D}(X)$ implies $a, b, c, d \ll X^{\frac{1}{4}}$, see Lemma \ref{lFD}. We will only consider $\beta$ satisfying the following five properties 
\begin{itemize}
\item $z_2, \Norm\gamma_2 \leq X^{\frac{1}{200}}$;
\item $\gcd(b - d, bc + cd) \leq \exp\left(\left(\log X\right)^{0.25}\right)$;
\item the $2$-adic valuation of $b - d$ is at most $\frac{\alpha}{2}$;
\item the squarefull part of $b - d$ is of size at most $\exp\left(\left(\log X\right)^{0.25}\right)$;
\item the odd, squarefree part of $2c^2 + (b - d)^2$ is at least $X^{\frac{99}{200}}$.
\end{itemize}
We claim that there are at most
\[
\ll \frac{X^{\frac{3}{4}}}{\exp\left(\left(\log X\right)^{0.2}\right)}
\]
elements $\beta$ that do not satisfy all five conditions. To do so, we shall bound the number of $\beta$ that fail a given bullet point in the above list. For the third and fourth bullet point this is easily verified. For the fifth bullet point, we use that $2c^2 + (b - d)^2$ represents a given integer at most $\ll_\epsilon X^{\frac{1}{4} + \epsilon}$ times, and this reduces the problem to an easy counting problem. A similar argument disposes with the first bullet point. Finally, for the second bullet point, we count the number of $\beta$ such that
\[
\gcd(b - d, b + d) > \exp\left(\frac{1}{2}\left(\log X\right)^{0.25}\right) \text{ or } \gcd(b - d, c) > \exp\left(\frac{1}{2}\left(\log X\right)^{0.25}\right).
\]
For those $\beta$, we bound the inner sum trivially by $\ll X^{\frac{1}{4}}/m$ inducing an error of size 
\[
\ll \frac{X}{\exp\left(\left(\log X\right)^{0.2}\right)}.
\]
For the remaining $\beta$, we have $G \ll_\epsilon X^\epsilon$ and hence $m \ll_\epsilon X^{\frac{1}{50} + \epsilon}$ by the first bullet point and the definition of $m$, see equation (\ref{eDefm}). Note that
\[
\mathbbm{1}_{\gcd(a + \beta, \sigma\tau(a + \beta)) = (1)} = \mathbbm{1}_{\gcd(a + \beta, \sigma\tau(\beta) - \beta) = (1)}.
\]
We use the M\"obius function to detect the coprimality conditions, which yields the following upper bound
\[
A(X, \mathfrak{d}, u_i, \rho) \ll \sum_\beta \sum_{f \in \Z/m\Z} \sum_{\mathfrak{d}_1 \mid \gamma_1} \sum_{\mathfrak{d}_2 \mid \sigma\tau(\beta) - \beta} \left|\sum_{\substack{a \in \Z \\ a \text{ sat. } (\ast\ast)}} \left(\frac{v_1}{2c^2 + (b - d)^2}\right)\right|,
\]
where $(\ast \ast)$ are the simultaneous conditions
\begin{align*}
a + \beta \in u_i\mathcal{D}(X), \quad &a + \beta \equiv 0 \bmod \mathfrak{d}, \quad a + \beta \equiv \rho \bmod 2^{10}, \quad a \equiv f \bmod m \\
&a + ci \equiv 0 \bmod \mathfrak{d_1}, \quad a + \beta \equiv 0 \bmod \mathfrak{d_2}.
\end{align*}
Define $m'$ to be the smallest positive integer that is divisible by $\text{lcm}(\mathfrak{d}, \mathfrak{d}_1, \mathfrak{d}_2)$. Put
\[
M := \text{lcm}\left(m, m'\right).
\]
The congruence conditions for $a$ in $(\ast \ast)$ are equivalent to at most one congruence condition modulo $M$. We assume that it is equivalent to exactly one congruence condition modulo $M$, say $F$, otherwise the inner sum is empty. Then we have
\begin{align}
\label{eBurgess2}
A(X, \mathfrak{d}, u_i, \rho) \ll \sum_\beta \sum_{f \in \Z/m\Z} \sum_{\mathfrak{d}_1 \mid \gamma_1} \sum_{\mathfrak{d}_2 \mid \sigma\tau(\beta) - \beta} \left|\sum_{\substack{a \in \Z \\ a + \beta \in u_i\mathcal{D}(X) \\ a \equiv F \bmod M}} \left(\frac{v_1}{2c^2 + (b - d)^2}\right)\right|.
\end{align}
We assume that $M \leq X^{\frac{1}{8}}$, since otherwise the trivial bound suffices. Furthermore, for fixed $\beta$, the condition $a + \beta \in u_i\mathcal{D}(X)$ means that $a$ runs over $\ll 1$ intervals with endpoints depending on $\beta$ and $u_i$. Since $a \ll X^{\frac{1}{4}}$, we know that each interval has length $\ll X^{\frac{1}{4}}$. We have the factorization 
\[
2c^2 + (b - d)^2 = q_1q_2,
\]
where $q_1$ is the odd, squarefree part. We know that $q_2 \ll X^{\frac{1}{200}}$, and we split the sum over $a$ in congruence classes modulo $q_2$. For fixed $b$, $c$ and $d$, the condition $a \equiv F \bmod M$ implies that $v_1$ is a linear function of $a$ with linear term not divisible by $q_1$ by our assumptions $q_1 \geq X^{\frac{99}{200}}$ and $M \leq X^{\frac{1}{8}}$. Indeed, $v_2$ and $t$ are determined by $F$, so this follows immediately from equation (\ref{efactorv}). Hence we may employ the Burgess bound \cite{Burgess} to equation (\ref{eBurgess2}) with $r =2$ and $q = q_1 \ll X^{\frac{1}{2}}$ to prove
\[
A(X, \mathfrak{d}, u_i, \rho) \ll_\epsilon X^{\frac{31}{32} + \frac{1}{50} + \frac{1}{200} + \epsilon} + X^{\frac{199}{200}} + X^{\frac{15}{16}} + \frac{X}{\exp\left(\left(\log X\right)^{0.2}\right)},
\]
where the second term accounts for the discarded congruence classes for $a$, the third term accounts for those $M$ with $M > X^{\frac{1}{8}}$ and the fourth term accounts for the discarded $\beta$. This establishes the following proposition.

\begin{prop}
\label{pSumI}
We have for all ideals $\mathfrak{d}$ of $\Z[\zeta_8]$
\[
A(X, \mathfrak{d}) \ll \frac{X}{\exp\left(\left(\log X\right)^{0.2}\right)}.
\]
\end{prop}

\section{Sums of type II}
\label{sSumII}
In equation (\ref{ewsplit}) we defined $[w]_1$ and $[w]_2$. We have the useful decomposition
\[
[w] = [w]_1 [w]_2.
\]
In this section we need to carefully study the multiplicative properties of $[w]$, and we do so by studying the multiplicative properties of $[w]_1$ and $[w]_2$. These properties will then be used to prove cancellation in sums of type II. We start by studying $[w]_1$; our treatment is almost identical to \cite{KM2}. If $w$ is an odd element of $\Z[\zeta_8]$, we have
\[
[w]_1 = \left(\frac{\left(\frac{1}{2} - \frac{1}{2\sqrt{2}}\right)\sigma(w)\sigma\tau(w)}{w}\right)_{4, M} = \left(\frac{\left(2 - \sqrt{2}\right)\sigma(w)\sigma\tau(w)}{w}\right)_{4, M}.
\]
Define
\begin{align}
\label{egamma1}
\gamma_1(w, z) := \left(\frac{\sigma(z)}{w}\right)_{2, M}.
\end{align}
For the remainder of this section, we use the convention that $\delta(w, z)$ is a function depending only on the congruence classes of $w$ and $z$ modulo $2^{10}$; at each occurence $\delta(w, z)$ may be a different function.

\begin{lemma}
\label{lMw1}
We have for all odd $w, z \in \Z[\zeta_8]$
\[
[wz]_1 = \delta(w, z) [w]_1 [z]_1 \gamma_1(w, z) \mathbbm{1}_{\gcd(w, \sigma\tau(z)) = (1)}.
\]
\end{lemma}

\begin{proof}
By definition of $[w]_1$ we have
\begin{align*}
[wz]_1 &= \left(\frac{\left(2 - \sqrt{2}\right)\sigma(wz)\sigma\tau(wz)}{wz}\right)_{4, M} \\
&= [w]_1 [z]_1 \left(\frac{\sigma(z)}{w}\right)_{4, M} \left(\frac{\sigma\tau(z)}{w}\right)_{4, M} \left(\frac{\sigma(w)}{z}\right)_{4, M} \left(\frac{\sigma\tau(w)}{z}\right)_{4, M}.
\end{align*}
Since $\sigma$ fixes $i$ and therefore any quartic residue symbol, Theorem \ref{t4R} yields
\begin{align*}
\left(\frac{\sigma(z)}{w}\right)_{4, M} \left(\frac{\sigma(w)}{z}\right)_{4, M} &= \delta(w, z) \left(\frac{\sigma(z)}{w}\right)_{4, M} \left(\frac{z}{\sigma(w)}\right)_{4, M} \\
&= \delta(w, z) \left(\frac{\sigma(z)}{w}\right)_{4, M} \sigma\left(\left(\frac{\sigma(z)}{w}\right)_{4, M}\right) \\
&= \delta(w, z) \left(\frac{\sigma(z)}{w}\right)_{2, M}.
\end{align*}
If we do the same computation for $\sigma\tau$, we obtain
\[
\left(\frac{\sigma\tau(z)}{w}\right)_{4, M} \left(\frac{\sigma\tau(w)}{z}\right)_{4, M} = \delta(w, z) \mathbbm{1}_{\gcd(w, \sigma\tau(z)) = (1)},
\]
since $\sigma\tau$ does not fix $i$. This proves the lemma.
\end{proof}

\noindent In the next lemma we collect the most important properties of $\gamma_1(w, z)$.

\begin{lemma}
\label{lPw1}
Let $w, z \in \Z[\zeta_8]$ be odd and define $\gamma_1(w, z)$ as in equation (\ref{egamma1}).
\begin{enumerate}
\item[(i)] $\gamma_1(w, z)$ is essentially symmetric
\[
\gamma_1(w, z) = \delta(w, z) \gamma_1(z, w).
\]
\item[(ii)] $\gamma_1(w, z)$ is multiplicative in both arguments
\[
\gamma_1(w, z_1z_2) = \gamma_1(w, z_1) \gamma_1(w, z_2), \quad \gamma_1(w_1w_2, z) = \gamma_1(w_1, z)\gamma_1(w_2, z).
\]
\end{enumerate}
\end{lemma}

\begin{proof}
This is straightforward.
\end{proof}

With this lemma we have completed our study of $[w]_1$ and $\gamma_1(w, z)$. We will now focus on $[w]_2$. Recall that
\[
[w]_2 = \left(\frac{2h}{g}\right) = \delta(w) \left(\frac{v}{u}\right).
\]
The second representation of $[w]_2$ is very convenient, since it allows us to use earlier work of Milovic \cite{Milovic2}. Define
\begin{align}
\label{egamma2}
\gamma_2(w, z) := \left(\frac{\sigma(wz)\sigma\tau(wz)}{w\tau(w)}\right)_{2, K},
\end{align}
where $K := \Q(\sqrt{2})$.

\begin{lemma}
\label{lMw2}
The following formula is valid for all odd $w, z \in \Z[\zeta_8]$
\[
[wz]_2 = \delta(w, z) [w]_2 [z]_2 \gamma_2(w, z).
\]
\end{lemma}

\begin{proof}
Milovic \cite[p.\ 1009]{Milovic2} defines the following symbol
\[
[u + v\sqrt{2}]_3 := \left(\frac{v}{u}\right).
\]
Then it is easily seen that $[w]_2 = \delta(w) [w \tau(w)]_3$ and that $w\tau(w)$ is totally positive. Now apply Proposition 8 of Milovic \cite{Milovic2}.
\end{proof}

To further our study of $\gamma_2(w, z)$, it will be convenient to define a second function $\text{m}(w)$ by the following formula
\[
\text{m}(w) := \gamma_2(w, 1) = \left(\frac{\sigma(w)\sigma\tau(w)}{w\tau(w)}\right)_{2, K}.
\]
It turns out that $\gamma_2(w, z)$ is neither symmetric nor multiplicative. Instead, it is symmetric and multiplicative twisted by the factor $\text{m}$.

\begin{lemma}
\label{lPw2}
Let $w, z \in \Z[\zeta_8]$ be odd and define $\gamma_2(w, z)$ as in equation (\ref{egamma2}).
\begin{enumerate}
\item[(i)] $\gamma_2(w, z)$ is twisted symmetric
\[
\gamma_2(w, z) \gamma_2(z, w) = \textup{m}(wz).
\]
\item[(ii)] $\gamma_2(w, z)$ is twisted multiplicative in $z$
\[
\gamma_2(w, z_1z_2) = \textup{m}(w) \gamma_2(w, z_1) \gamma_2(w, z_2).
\]
\end{enumerate}
\end{lemma}

\begin{proof}
Left to the reader.
\end{proof}

With this out of the way we are ready to tackle the sums of type II. Let $\{\alpha_w\}$ and $\{\beta_z\}$ be sequences of complex numbers of absolute value at most $1$ and let $\rho$ and $\mu$ be invertible congruence classes modulo $2^{10}$. We define
\[
B_1(M, N, \rho, \mu) := \sum_{\substack{w \in \mathcal{D}(M) \\ w \equiv \rho \bmod 2^{10}}} \sum_{\substack{z \in \mathcal{D}(N) \\ z \equiv \mu \bmod 2^{10}}} \alpha_w \beta_z [wz],
\]
where we suppress the dependence on $\{\alpha_w\}$ and $\{\beta_z\}$. Then we have the following proposition.

\begin{prop}
\label{pBil}
There is an absolute constant $\theta_3 > 0$ such that for all sequences of complex numbers $\{\alpha_w\}$ and $\{\beta_z\}$ of absolute value at most $1$, all invertible congruence classes $\rho$ and $\mu$ modulo $2^{10}$
\[
B_1(M, N, \rho, \mu) \ll \left(M^{-\frac{1}{24}} + N^{-\frac{1}{24}}\right) MN (\log MN)^{\theta_3}.
\]
\end{prop}

\begin{proof}
We start by expanding $[wz]$ using Lemma \ref{lMw1} and Lemma \ref{lMw2}. We may absorb $[w]_1$, $[w]_2$, $[z]_1$ and $[z]_2$ in the coefficients $\alpha_w$ and $\beta_z$. Then it suffices to prove for all sequences of complex numbers $\{\alpha_w\}$ and $\{\beta_z\}$ of absolute value at most $1$ and all invertible congruence classes $\rho$ and $\mu$ modulo $2^{10}$ the following estimate
\begin{align*}
B_2(M, N, \rho, \mu) &:= \sum_{\substack{w \in \mathcal{D}(M) \\ w \equiv \rho \bmod 2^{10}}} \sum_{\substack{z \in \mathcal{D}(N) \\ z \equiv \mu \bmod 2^{10}}} \alpha_w \beta_z \gamma_1(w, z) \gamma_2(w, z) \mathbbm{1}_{\gcd(w, \sigma\tau(z)) = (1)} \\
&\ll \left(M^{-\frac{1}{24}} + N^{-\frac{1}{24}}\right) MN (\log MN)^{\theta_3}.
\end{align*}
Define
\[
\gamma_3(w, z) := \left(\frac{\sigma(z)\sigma\tau(z)}{w\tau(w)}\right)_{2, K},
\]
so that we have the factorization $\gamma_2(w, z) = \text{m}(w) \gamma_3(w, z)$. Absorbing $\text{m}(w)$ in $\alpha_w$ and using the identity
\[
\gamma_3(w, z) \mathbbm{1}_{\gcd(w, \sigma\tau(z)) = (1)} = \gamma_3(w, z),
\]
we see that it is enough to establish
\begin{align*}
B_3(M, N, \rho, \mu) &:= \sum_{\substack{w \in \mathcal{D}(M) \\ w \equiv \rho \bmod 2^{10}}} \sum_{\substack{z \in \mathcal{D}(N) \\ z \equiv \mu \bmod 2^{10}}} \alpha_w \beta_z \gamma_1(w, z) \gamma_3(w, z) \\
&\ll \left(M^{-\frac{1}{24}} + N^{-\frac{1}{24}}\right) MN (\log MN)^{\theta_3}.
\end{align*}
Theorem \ref{t2R} shows that $\gamma_3(w, z)$ is also essentially symmetric, i.e.
\[
\gamma_3(w, z) = \delta(w, z)\gamma_3(z, w).
\]
Due to the symmetry of $\gamma_1(w, z)$, see Lemma \ref{lPw1}(i), and the symmetry of $\gamma_3(w, z)$, we may further reduce to the case $N \geq M$. We take $k := 12$ and apply H\"older's inequality with $1 = \frac{k - 1}{k} + \frac{1}{k}$ to the $w$ variable to obtain
\[
\left|B_3(M, N, \rho, \mu)\right|^k \leq
\left(\sum_{\substack{w \in \mathcal{D}(M) \\ w \equiv \rho \bmod 2^{10}}} \left|\alpha_w\right|^{\frac{k}{k - 1}}\right)^{k - 1} \sum_{\substack{w \in \mathcal{D}(M) \\ w \equiv \rho \bmod 2^{10}}} \left|\sum_{\substack{z \in \mathcal{D}(N) \\ z \equiv \mu \bmod 2^{10}}} \beta_z \gamma_1(w, z) \gamma_3(w, z)\right|^k.
\]
The first factor is trivially bounded by $\ll M^{k - 1}$ with absolute implied constant. Lemma \ref{lPw1}(ii) implies that $\gamma_1(w, z)$ is multiplicative in $z$ and Lemma \ref{lPw2}(ii) implies that $\gamma_3(w, z)$ is multiplicative in $z$. Hence $\gamma_1(w, z) \gamma_3(w, z)$ is multiplicative in $z$. We conclude that
\begin{align}
\label{eHolder}
\left|B_3(M, N, \rho, \mu)\right|^k \ll M^{k - 1} \sum_{\substack{w \in \mathcal{D}(M) \\ w \equiv \rho \bmod 2^{10}}} \epsilon(w) \sum_z \beta_z' \gamma_1(w, z) \gamma_3(w, z),
\end{align}
where
\[
\epsilon(w) := \left(\frac{\left|\sum_{\substack{z \in \mathcal{D}(N) \\ z \equiv \mu \bmod 2^{10}}} \beta_z \gamma_1(w, z) \gamma_3(w, z)\right|}{\sum_{\substack{z \in \mathcal{D}(N) \\ z \equiv \mu \bmod 2^{10}}} \beta_z \gamma_1(w, z) \gamma_3(w, z)}\right)^k
\]
and
\[
\beta_z' := \sum_{\substack{z = z_1 \cdot \ldots \cdot z_k \\ z_1, \ldots, z_k \in \mathcal{D}(N) \\ z_1 \equiv \ldots \equiv z_k \equiv \mu \bmod 2^{10}}} \beta_{z_1} \cdot \ldots \cdot \beta_{z_k}.
\]
We will now study the summation condition for $z$ in the inner sum of equation (\ref{eHolder}) more carefully. By construction, $\mathcal{D}(N)$ contains exactly eight generators of any principal ideal. Furthermore, there are $\ll N^k$ values of $z$ for which $\beta_z' \neq 0$. Hence we obtain the bound
\[
\sum_z \left(\beta_z'\right)^2 \ll (\log N)^{\theta_3} N^k
\]
for some absolute constant $\theta_3$, since $k$ is fixed.  An application of the Cauchy-Schwarz inequality over the $z$ variable yields
\begin{multline}
\label{eCauchy}
\left(\sum_{\substack{w \in \mathcal{D}(M) \\ w \equiv \rho \bmod 2^{10}}} \hspace{-0.5cm} \epsilon(w) \sum_z \beta_z' \gamma_1(w, z) \gamma_3(w, z)\right)^2 = \left(\sum_z \beta_z' \hspace{-0.3cm} \sum_{\substack{w \in \mathcal{D}(M) \\ w \equiv \rho \bmod 2^{10}}} \hspace{-0.5cm} \epsilon(w) \gamma_1(w, z) \gamma_3(w, z)\right)^2 \\
\ll (\log N)^{\theta_3} N^k \sum_{\substack{w_1 \in \mathcal{D}(M) \\ w_1 \equiv \rho \bmod 2^{10}}} \sum_{\substack{w_2 \in \mathcal{D}(M) \\ w_2 \equiv \rho \bmod 2^{10}}} \hspace{-0.5cm} \epsilon(w_1) \overline{\epsilon(w_2)} \sum_z \gamma_1(w_1w_2, z) \gamma_3(w_1w_2, z),
\end{multline}
because $\gamma_1(w, z)$ and $\gamma_3(w, z)$ are multiplicative in $w$. Conveniently, inequality (\ref{eCauchy}) remains valid if we extend the sum over $z$ to a larger domain. Let $z_1, \ldots, z_k \in \mathcal{D}(N)$ and write 
\[
z_i = \sum_{j = 1}^4 a_{ij} \zeta_8^j.
\]
Then we have $|a_{ij}| \ll N^{\frac{1}{4}}$. Now define
\[
\mathcal{B}(C) := \left\{\sum_{j = 1}^4 a_j \zeta_8^j : a_j \in \Z, |a_j| \leq CN^{\frac{k}{4}}\right\}.
\]
Then, if $C$ is sufficiently large, $\beta'_z \neq 0$ implies $z \in \mathcal{B}(C)$. For this choice of $C$, we extend the range of summation over $z$ in equation (\ref{eCauchy}) to the set $\mathcal{B}(C)$. We split the sum over $z$ in congruence classes $\zeta$ modulo $\Norm(w_1w_2)$; we claim that for all odd $w$ 
\[
\sum_{\zeta \bmod \Norm(w)} \gamma_1(w, \zeta) \gamma_3(w, \zeta) = 0
\]
provided that $\Norm(w)$ is not squarefull. Substituting the definition of $\gamma_1(w, \zeta)$ and $\gamma_3(w, \zeta)$ gives
\[
f(w) := \sum_{\zeta \bmod \Norm(w)} \gamma_1(w, \zeta) \gamma_3(w, \zeta) = \sum_{\zeta \bmod \Norm(w)} \left(\frac{\sigma(\zeta)\sigma\tau(\zeta)}{w\tau(w)}\right)_{2, K} \left(\frac{\sigma(\zeta)}{w}\right)_{2, M}.
\]
Then a calculation shows that for all odd $w$ and $w'$ satisfying $(\Norm(w), \Norm(w')) = 1$
\[
f(ww') = f(w)f(w').
\]
Hence, to establish the claim, it is enough to prove that $f(w) = 0$ if $w$ is an odd prime of degree $1$. To do so, we start with the identity
\[
\left(\frac{\sigma(\zeta)\sigma\tau(\zeta)}{w\tau(w)}\right)_{2, K} = \left(\frac{\sigma(\zeta)\sigma\tau(\zeta)}{w}\right)_{2, M}.
\]
Here we rely in an essential way that $w$ is an odd prime of degree $1$, so we have an isomorphism of finite fields $O_M/w \cong O_K/w\tau(w)$.  We use this to give a simple expression for $f(w)$
\[
f(w) = \sum_{\zeta \bmod \Norm(w)} \left(\frac{\sigma\tau(\zeta)}{w}\right)_{2, M} \mathbbm{1}_{(\sigma(\zeta), w) = (1)},
\]
which apart from a non-zero factor is
\[
\sum_{\zeta \bmod \sigma(w)\sigma\tau(w)} \left(\frac{\sigma\tau(\zeta)}{w}\right)_{2, M} \mathbbm{1}_{(\sigma(\zeta), w) = (1)} = 
\sum_{\zeta \bmod \sigma\tau(w)} \left(\frac{\sigma\tau(\zeta)}{w}\right)_{2, M} \sum_{\zeta \bmod \sigma(w)} \mathbbm{1}_{(\sigma(\zeta), w) = (1)} = 0.
\]
Note that $\sigma(w)$ and $\sigma\tau(w)$ are coprime, so that we are allowed to expand the sum over $\sigma(w) \sigma\tau(w)$ as the product of the two sums over $\sigma(w)$ and $\sigma\tau(w)$. With the claim established, we can give an upper bound for the sum over $z \in \mathcal{B}(C)$
\[
\sum_{z \in \mathcal{B}(C)} \gamma_1(w_1w_2, z) \gamma_3(w_1w_2, z) \ll
\left\{
	\begin{array}{ll}
		N^k  & \mbox{if } \Norm(w_1w_2) \text{ is squarefull} \\
		\sum_{i = 1}^4 M^{2i} N^{k\left(1 - \frac{i}{4}\right)} & \mbox{otherwise,}
	\end{array}
\right.
\]
where the second bound uses the claim and $\Norm(w_1w_2) \leq M^2$. Because of our choice of $k$ and $N \geq M$, we can simplify the second bound to $M^2 N^{\frac{3}{4}k}$. Equation (\ref{eHolder}), equation (\ref{eCauchy}) and the above bound acquire the validity of
\begin{align*}
\left|B_3(M, N, \rho, \mu)\right|^{2k} &\ll (\log N)^{\theta_3} M^{2k - 2}N^k\left(M \cdot N^k + M^2 \cdot M^2N^{\frac{3}{4}k}\right) \\
&\ll (\log N)^{\theta_3} \left(M^{2k - 1} \cdot N^k + M^{2k + 2} \cdot N^{\frac{7}{4}k}\right).
\end{align*}
Since the first term above dominates the second term due to our choice of $k$ and $N \geq M$, the proof of the proposition is complete.
\end{proof}

Having dealt with sums of type II for the symbol $[wz]$, we now turn to sums of type II with $a_{\mathfrak{m}\mathfrak{n}}$. For sequences of complex numbers $\{\alpha_{\mathfrak{m}}\}$ and $\{\beta_{\mathfrak{n}}\}$ of absolute value at most $1$ we defined in Section \ref{sSieve} the following sum
\[
B(M, N) = \sum_{\Norm{\mathfrak{m}} \leq M} \sum_{\Norm{\mathfrak{n}} \leq N} \alpha_{\mathfrak{m}} \beta_{\mathfrak{n}} a_{\mathfrak{m}\mathfrak{n}}.
\]

\begin{prop}
\label{pSumII}
There is an absolute constant $\theta_3 > 0$ such that for all sequences of complex numbers $\{\alpha_{\mathfrak{m}}\}$ and $\{\beta_{\mathfrak{n}}\}$ of absolute value at most $1$
\[
B(M, N) \ll \left(M^{-\frac{1}{24}} + N^{-\frac{1}{24}}\right) MN (\log MN)^{\theta_3}.
\]
\end{prop}

\begin{proof}
By picking generators for $\mathfrak{m}$ and $\mathfrak{n}$ we obtain the following identity
\[
B(M, N) = \sum_{\Norm{\mathfrak{m}} \leq M} \sum_{\Norm{\mathfrak{n}} \leq N} \alpha_{\mathfrak{m}} \beta_{\mathfrak{n}} a_{\mathfrak{m}\mathfrak{n}} = \frac{1}{64} \sum_{w \in \mathcal{D}(M)} \sum_{z \in \mathcal{D}(N)} \alpha_w \beta_z a_{(wz)}.
\]
We split the sum $B(M, N)$ in congruence classes modulo $2^{10}$. We need only consider invertible congruence classes, since otherwise $a_{wz} = 0$ by definition. Furthermore, condition (\ref{e8p}) depends only on $g$ modulo $4$, which is in turn determined by $w$ modulo $4$. Therefore, it suffices to bound the following sum
\[
\sum_{\substack{w \in \mathcal{D}(M) \\ w \equiv \rho \bmod 2^{10}}} \sum_{\substack{z \in \mathcal{D}(N) \\ z \equiv \mu \bmod 2^{10}}} \alpha_w \beta_z \left([wz] + [\epsilon wz] + [\epsilon^2 wz] + [\epsilon^3 wz]\right),
\]
where $\rho$ and $\mu$ are invertible congruence classes modulo $2^{10}$ such that $g \equiv 1 \bmod 4$. From Lemma \ref{lMw1} and Lemma \ref{lMw2} we deduce that
\[
[\epsilon wz] = \delta(w, z) [\epsilon] [wz].
\]
Now apply Proposition \ref{pBil}.
\end{proof}

\end{document}